\documentclass{amsart}
\usepackage{amsmath}
\usepackage{amsfonts}
\usepackage{amssymb}
\usepackage{amsthm}

\theoremstyle{plain}
\newtheorem{theorem}{Theorem}[section]
\newtheorem{conjecture}[theorem]{Conjecture}
\newtheorem{corollary}[theorem]{Corollary}
\newtheorem{proposition}[theorem]{Proposition}
\newtheorem{lemma}[theorem]{Lemma}

\theoremstyle{definition}
\newtheorem{definition}[theorem]{Definition}
\newtheorem{remark}[theorem]{Remark}
\newtheorem{example}[theorem]{Example}

\newtheorem{convention}[theorem]{Convention}

\numberwithin{equation}{section}

\input xy
\xyoption{all}

\usepackage{graphics}

\begin{document}

\title[Enumeration of quasigroup words]{On the enumeration and asymptotic growth of free quasigroup words}
\author[J.D.H.~Smith]{Jonathan D.H.~Smith$^1$}
\author[S.G.~Wang]{Stefanie G.~Wang$^2$}
\address{$^1$ Dept. of Mathematics\\
Iowa State University\\
Ames, Iowa 50011, U.S.A.}
\address{$^1$ Dept. of Mathematics \& Statistics\\
Smith College\\
Northampton, Massachusetts 01063, U.S.A.}
\email{$^1$jdhsmith@iastate.edu\phantom{,}}
\email{$^2$stwang@smith.edu}



\begin{abstract}
The paper counts the number of reduced quasigroup words of a particular length in a certain number of generators. Taking account of the relationship with the Catalan numbers, counting words in a free magma, we introduce the term peri-Catalan number for the free quasigroup word counts. The main result of the paper is an exact recursive formula for the peri-Catalan numbers, structured by the Euclidean Algorithm. 

The Euclidean Algorithm structure does not readily lend itself to standard techniques of asymptotic analysis. However, conjectures for the asymptotic behavior of the peri-Catalan numbers, substantiated by numerical data, are presented. A remarkable aspect of the observed asymptotic behavior is the so-called asymptotic irrelevance of quasigroup identities, whereby cancelation resulting from quasigroup identities has a negligible effect on the asymptotic behavior of the peri-Catalan numbers for long words in a large number of generators.  
\end{abstract}

\maketitle

%\tableofcontents

\section{Introduction}

In the free magma on one generator $x$, the number of words of length $n$ (i.e., with $n$ occurrences of the letter $x$) is given by the Catalan number $C_n$ (in the natural leaf-counting indexation of \cite[Rem.~5.5]{CL}). More generally, the number of words of length $n$ in the free magma on $s$ generators is given as $s^nC_n$. Now consider the free quasigroup on $s$ generators. The word problem was solved by Evans in a notable paper which gave one of the earliest applications of confluence (the ``Diamond Lemma") in rewriting systems \cite{Evans51} \cite[\S I.3]{Evans90}. We now quantify Evans' normal form theorem by asking for the number of reduced words of length $n$, which we define as the \emph{peri-Catalan number} $P^s_n$ (\S\ref{SS:p-Cdefnd}). In the language of geometric group theory, we are interested in the growth rate of a free quasigroup (compare \cite{Grigorchuk}).

In our terminology, the Greek qualifier $\pi\varepsilon\rho\iota$ has the sense of ``around" (as in ``perimeter"), in reference to the fact that the free quasigroup sits around the free magma. The close relationship between the Catalan numbers and the peri-Catalan numbers is reinforced by our conjectured formula
\begin{equation}\label{E:asnPtotc}
\lim_{n\to\infty}
\frac{\log P^s_n}{\log C_n+n\log 3s-\log3}
\simeq 1-\frac{0.019\dots}{s-\log\big((1+\sqrt5)/2\big)}
\end{equation}
for the asymptotic behavior of the peri-Catalan numbers. We do not have an interpretation for the numerical constant $0.019\dots$ appearing in the numerator of the fraction on the right hand side of \eqref{E:asnPtotc}.

The main result of the paper is the recursive formula \eqref{E:FormuPsn} for the peri-Catalan numbers $P^s_n$. The formula is structured around the Euclidean Algorithm. More specifically, for each partition $n=(n-k)+k$ of length two, with $n-k\ge k$, and for each call to the Division Algorithm within the Euclidean Algorithm for computing $\gcd(n,k)$, there is a corresponding summand in the recursion formula. The appearance of the golden ratio in the conjectured formula \eqref{E:asnPtotc} may be related to the fact that the Fibonacci numbers provide the worst cases for runs of the Euclidean Algorithm \cite[\S4.5.3, Th,~F]{Knuth2}.

Section~\ref{S:background} gives the background definitions for equationally-defined quasigroups, as algebras $(Q,\cdot,/,\backslash)$ with basic operations of multiplication $\cdot$, right division $/$, and left division $\backslash$ that satisfy identities such as $x\cdot(x\backslash y)=y$. This identity corresponds to one of the Latin square properties for the body of the multiplication table of a finite quasigroup, specifically that each element $y$ appears in the row labeled by $x$, namely in the column labeled by $x\backslash y$ \cite[Defn.~11.6(a)]{ITAA}. Thus $x\cdot(x\backslash y)$ is not reduced as a quasigroup word of length $3$, and will not contribute to the peri-Catalan number $P^2_3$.

Section~\ref{S:triality} recalls the triality symmetry ($S_3$-action) of the language of quasigroups. The recursion formula for $P^s_n$ is based on an inductive process creating a quasigroup word of length $n$ by applying one of the three basic operations to a pair of reduced words whose lengths sum to $n$ (\S\ref{SS:IndProces}). The triality notation of Section~\ref{S:triality} is used in \S\ref{S:rootcancellations} to give an efficient and non-repetitive treatment of the cancelations that take place in the inductive process. This application of triality is comparable to its use in the streamlining of the proof of Evans' normal form theorem \cite{Sm131}. The Euclidean Algorithm notation (generally following \cite[\S1.6]{ITAA}) is introduced in \S\ref{SS:EucliAlg}. The double inductive proof of the main Theorem~\ref{T:ExactPsn} is then given in \S\ref{SS:RecProPC}.

The recursive formula that is provided by Theorem~\ref{T:ExactPsn} does not appear to be readily amenable to the standard methods of asymptotic analysis. Nevertheless, the concluding Section~\ref{S:Asympttc} provides some conjectured results based on an alternative approach. In particular, Conjecture~\ref{Cj:DepGeNum} suggests the asymptotic behavior \eqref{E:asnPtotc} for all generator counts $s$. The so-called \textit{asymptotic irrelevance of quasigroup identities} (Conjecture~\ref{Cj:AsIrQpId}) asserts that cancelation resulting from quasigroup identities has a negligible effect on the asymptotic behavior of the peri-Catalan numbers for long words in a large number of generators.

We use algebraic notation (with functions following, or as superfixes of, their arguments) as the default. Note that $S_r$ denotes the symmetric group on $r$ letters.

\section{Background}\label{S:background}

\subsection{Magmas, rooted binary trees, and Catalan numbers}\label{SS:MRBTCaNo}

A \emph{magma} $(M,\cdot)$ is a set equipped with a single binary operation $M\times M\to M;(m_1,m_2)\mapsto m_1\cdot m_2$. Elements of free magmas are described as (\emph{magma}) \emph{words}. Words in the free magma on a singleton alphabet $\{a\}$ are repeated concatenations of the generator $a$ under the magma operation. Each word $w$ is determined by a \emph{rooted binary tree} $B_w$, defined recursively as follows:
\begin{enumerate}
\item[$(\mathrm a)$]
The tree $B_{a}$ consists of a single node (which is both root and leaf);
\item[$(\mathrm b)$]
For words $u,v$, the tree $B_{u\cdot v}$ has a root with $B_u$ as a left child, and $B_v$ as a right child.
\end{enumerate}

If $B_w$ has $n$ leaves, then the magma word $w$ has \emph{length} $n$. Although varying conventions are encountered, we define the $n$-th \emph{Catalan number} $C_n$ as the number of magma words of length $n$ in the single generator $a$ (compare \cite[\S5]{CL}). Then the free magma in $s$ generators has $s^nC_n$ words of length $n$.

A magma $(Q,\cdot)$ is a \emph{combinatorial quasigroup} if knowing any two arguments in the equation $x\cdot y=z$ uniquely specifies the third argument for all $x, y, z\in Q$.

\subsection{Quasigroups}\label{SS:quasigps}

An \emph{equational quasigroup} $(Q, \cdot,/,\backslash)$ is a set equipped with three \emph{basic} binary operations, multiplication $\cdot$, right division /, and left division $\backslash$ such that for all $x, y\in Q$, the following identities are satisfied:
\begin{equation}\label{E:QgpIdens}
 \begin{array}{lll}
( \mbox{SL} ) & x\cdot(x\backslash y) = y \, ;\qquad
( \mbox{SR} ) & y = (y/x)\cdot x \, ; \\
( \mbox{IL} ) & x\backslash (x\cdot y) = y \, ;\qquad
( \mbox{IR} ) & y = (y\cdot x)/x \, .
 \end{array}
\end{equation}
Note that (IL), (IR) give the respective injectivity of the left multiplication
$$
L(x)\colon Q\to Q;y\mapsto xy
$$
and right multiplication
$$
R(x)\colon Q\to Q;y\mapsto yx
$$
for $x\in Q$, while (SL), (SR) give their surjectivity. Thus an equational quasigroup $(Q,\cdot,/,\backslash)$ yields a combinatorial quasigroup $(Q,\cdot)$. Conversely, a combinatorial quasigroup $(Q,\cdot)$ yields an equational quasigroup $(Q,\cdot,/,\backslash)$ with $x/y=xR(y)^{-1}$ and $x\backslash y=yL(x)^{-1}$.

In an equational quasigroup $(Q,\cdot,/,\backslash)$, the three equations
\begin{equation}\label{E:firsthre}
x_1\cdot x_2=x_3\, ,\qquad x_3/x_2=x_1\, ,\qquad x_1\backslash x_3=x_2
\end{equation}
involving the basic operations are equivalent. Introducing the \emph{opposite} operations
$$
x\circ y=y\cdot x\, ,\qquad x/\negthinspace/y=y/x\, ,\qquad x\backslash\negthinspace\backslash y=y\backslash x
$$
on $Q$, the equations \eqref{E:firsthre} are further equivalent to the equations
$$
x_2\circ x_1=x_3\, ,\qquad x_2/\negthinspace/x_3=x_1\, ,\qquad x_3\backslash\negthinspace\backslash x_1=x_2\, .
$$
Thus each of the basic and opposite operations
\begin{equation}\label{E:conjugac}
(Q,\cdot),\quad(Q,/),\quad(Q,\backslash),
\quad(Q,\circ),\quad(Q,/\negthinspace/),\quad(Q,\backslash\negthinspace\backslash)
\end{equation}
forms a (combinatorial) quasigroup. In particular, note that the identities (IR) in $(Q,\backslash)$ and (IL) in (Q,/) yield the respective identities
$$
 \begin{array}{ll}
( \mbox{DL} ) & x/(y\backslash x)=y \, ,\index{DL@(DL)}\\
( \mbox{DR} ) & y=(x/y)\backslash x \index{DR@(DR)}
 \end{array}
$$
in the basic quasigroup divisions. The six quasigroups \eqref{E:conjugac} are known as the \emph{conjugates}, ``parastrophes'' \cite{Sade57a} or ``derived quasigroups'' \cite{James} of $(Q,\cdot)$.

\subsection{Basic words and parsing trees}\label{SS:BasParTr}

In the free quasigroup on an alphabet
$$
\{a_1, a_2, \dots, a_s\}
$$
of $s$ letters, (\emph{basic}) \emph{quasigroup words} are repeated concatenations of the generators under the three basic quasigroup operations $\cdot,/,\backslash$. Each basic quasigroup word $w$ is in one-one correspondence with a \emph{basic} \emph{parsing tree} $T_w$, defined recursively as follows:
\begin{enumerate}
\item[$(\mathrm a)$]
For $1\le i\le s$, the tree $T_{a_i}$ is a single vertex annotated by $a_i$;
\item[$(\mathrm b)$]
For basic words $u,v$, the tree $T_{u\cdot v}$ has:
\begin{enumerate}
\item[(i)]
a root annotated by the multiplication,
\item[(ii)]
$T_u$ as a left child, and $T_v$ as a right child;
\end{enumerate}
\item[$(\mathrm c)$]
For basic words $u,v$, the tree $T_{u/v}$ has:
\begin{enumerate}
\item[(i)]
a root annotated by the right division,
\item[(ii)]
$T_u$ as a left child, and $T_v$ as a right child;
\end{enumerate}
\item[$(\mathrm d)$]
For basic words $u,v$, the tree $T_{u\backslash v}$ has:
\begin{enumerate}
\item[(i)]
a root annotated by the left division,
\item[(ii)]
$T_u$ as a left child, and $T_v$ as a right child.
\end{enumerate}
\end{enumerate}

\begin{remark}\label{R:WdLength}
(a) A basic parsing tree stripped of its annotation yields a binary rooted tree.
\vskip 1.5mm
\noindent
(b) If a basic parsing tree $T_u$ has $n$ leaves, then the basic quasigroup word $u$ has \emph{length} $n$, involving $n-1$ basic operations determined by the annotations of the nodes of $T_u$.
\end{remark}

\section{Triality}\label{S:triality}

\subsection{Triality action of the symmetric group $S_3$}

The symmetric group $S_3$ on the $3$-element set $\{1,2,3\}$ is presented as
$$
\big\langle \sigma,\tau \,\big\vert\, \sigma^2=\tau^2=(\sigma\tau)^3=1 \big\rangle
$$
writing $\sigma$ and $\tau$ for the respective transpositions $(12)$ and $(23)$. The Cayley diagram of the presentation is
\begin{equation}\label{E:cayls3st}
\begin{matrix}
1  &\Longleftrightarrow   &\tau   &\longleftrightarrow     &\tau\sigma  \\
\updownarrow   &            &       &                   &\Updownarrow   \\
\sigma &\Longleftrightarrow   &\sigma\tau  &\longleftrightarrow     &\sigma\tau\sigma
\end{matrix}
\end{equation}
with $\leftrightarrow$ for right multiplication by $\sigma$ and $\Leftrightarrow$ for right multiplication by $\tau$.

Now consider the full set
\begin{equation}\label{E:conjugat}
\{\cdot, \backslash, /\negthinspace/, /, \backslash\negthinspace\backslash, \circ\}
\end{equation}
of all quasigroup operations, both basic and opposite. For current purposes, it will be convenient to use postfix notation for binary operations, setting $x\cdot y=xy\,\mu$ and rewriting the first equation of \eqref{E:firsthre} in the form
\begin{equation}\label{E:x1x2mux3}
x_1x_2\,\mu =x_3 \, .
\end{equation}
The full set \eqref{E:conjugat} is construed as the homogeneous space
\begin{equation}\label{E:homospop}
\mu^{S_3}=\{\mu^g\mid g\in S_3\}
\end{equation}
for a regular right permutation action of the symmetric group $S_3$, such that \eqref{E:x1x2mux3} is equivalent to
\begin{equation}\label{E:x1gx2gmg}
x_{1g}x_{2g}\mu^g=x_{3g}
\end{equation}
for each $g$ in $S_3$. This action is known as \emph{triality}.\footnote{Triality in this sense is not to be confused with the distinct but related notion ultimately arising from the action of $S_3$ on the Dynkin diagram $D_4$ (compare, say, \cite{HaNa}).} The six binary operations, in their positions corresponding to the Cayley diagram \eqref{E:cayls3st}, are displayed in Figure~\ref{F:trialops}.
\vskip .2in
\begin{figure}[bht]
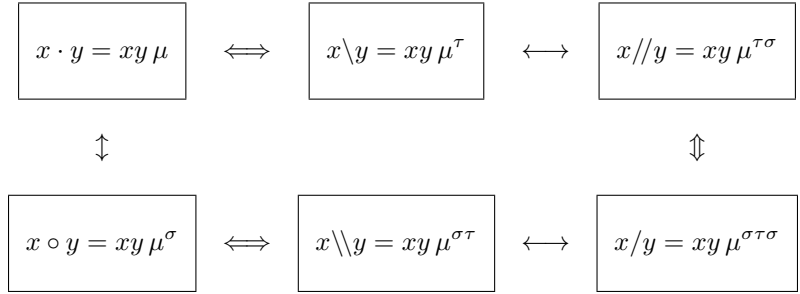

$$
\begin{matrix}
\boxed{\phantom{\Bigg|}x\cdot y=xy\,\mu\phantom{\Bigg|}}
&\Longleftrightarrow
&\boxed{\phantom{\Bigg|}x\backslash y=xy\,\mu^{\tau}\phantom{\Bigg|}}
&\longleftrightarrow
&\boxed{\phantom{\Bigg|}x/\negthinspace/y=xy\,\mu^{\tau\sigma}\phantom{\Bigg|}}
\\
&            &       &                   & \\
\updownarrow   &            &       &                   &\Updownarrow   \\
&            &       &                   & \\
\boxed{\phantom{\Bigg|}x\circ y=xy\,\mu^{\sigma}\phantom{\Bigg|}}
&\Longleftrightarrow
&\boxed{\phantom{\Bigg|}x\backslash\negthinspace\backslash y=xy\,\mu^{\sigma\tau}\phantom{\Bigg|}}
&\longleftrightarrow
&\boxed{\phantom{\Bigg|}x/y=xy\,\mu^{\sigma\tau\sigma}\phantom{\Bigg|}}
\end{matrix}
$$
\vskip .3in
\caption{Symmetry of the quasigroup operations.
}\label{F:trialops}
\end{figure}
In the figure, the opposite of each operation $\mu^g$ is given by $\mu^{\sigma g}$, so passage to the opposite operation corresponds to left multiplication by the transposition $\sigma$ in the symmetric group $S_3$. The three pairs of opposite operations lie in the respective columns of Figure~\ref{F:trialops}.

\begin{remark}
A deeper interpretation of Figure~\ref{F:trialops}, and the relationship between the basic and the opposite quasigroup operations, is provided by the \emph{orthant structure} discussed in \cite[Ex.~12.1]{Sm153}. The basic operations constitute one set of representatives for the three doubleton orthants $\Omega^{+,+},\Omega^{-,+},\Omega^{+,-}$, while the opposite operations are the three remaining elements of the respective orthants.
\end{remark}

\subsection{Invertible elements of the monoid of binary words}

Just as the left multiplication by $\sigma$ in the symmetric group $S_3$ has a simple interpretation, so does the left multiplication by $\tau$. Let $M$ be the complete set of all derived binary operations on a quasigroup. Together, these operations constitute the free algebra on two generators $x$, $y$ in the variety $\mathbf{Q}$ of quasigroups. A multiplication $*$ is defined on $M$ by
\begin{equation}\label{E:mltonbin}
xy(\alpha*\beta)=x\,xy\alpha\,\beta\, .
\end{equation}
The right projection $xy\epsilon=y$ also furnishes a binary operation $\epsilon$.

\begin{lemma}
The set $M$ of all derived binary quasigroup operations forms a monoid $(M,*,\epsilon)$ under the multiplication \eqref{E:mltonbin}, with identity element $\epsilon$.
\end{lemma}

\begin{proof}
Observe that
$$
xy(\alpha*\epsilon)=x\,xy\alpha\,\epsilon=xy\alpha
\quad
\mbox{ and }
\quad
xy(\epsilon*\alpha)=x\,xy\epsilon\,\alpha=xy\alpha
$$
for $\alpha$ in $M$, so $\epsilon$ is an identity element. Consider $\alpha$, $\beta$, $\gamma$ in $M$. Then
\begin{align*}
xy\big((\alpha*\beta)*\gamma\big)&=x\,xy(\alpha*\beta)\,\gamma
\\
&
=xxxy\alpha\beta\gamma
\\
&
=x\,xy\alpha(\beta*\gamma)=xy\big(\alpha*(\beta*\gamma)\big) \, ,
\end{align*}
confirming the associativity of the multiplication \eqref{E:mltonbin}.
\end{proof}

The significance of the left multiplication by $\tau$ then follows.

\begin{proposition}\label{T:hyperqgp}
For each element $g$ of $S_3$, the binary operation $\mu^g$ is a unit of the monoid $M$, with inverse $\mu^{\tau g}$.
\end{proposition}

\begin{proof}
The (IL) identity $x\backslash(x\cdot y)=y$ becomes $x\,xy\mu\,\mu^{\tau}=y$ or $\mu*\mu^{\tau}=\epsilon$. Similarly the (SL) identity $x\cdot(x\backslash y)=y$ becomes $x\,xy\mu^{\tau}\,\mu=x$ or $\mu^{\tau}*\mu=\epsilon$. This means that $\mu$ and $\mu^{\tau}$ are mutual inverses.

The (IR) identity $(y\cdot x)/x=y$ is $x/\negthinspace/(x\circ y)=y$. This becomes $x\,xy\mu^{\sigma}\,\mu^{\tau\sigma}=y$ or $\mu^{\sigma}*\mu^{\tau\sigma}=\epsilon$. Similarly (SR), namely $(y/ x)\cdot x=y$, may be written as $x\circ(x/\negthinspace/ y)=y$. This becomes $x\,xy\mu^{\tau\sigma}\,\mu^{\sigma}=y$ or $\mu^{\tau\sigma}*\mu^{\sigma}=\epsilon$. Thus $\mu^{\sigma}$ and $\mu^{\tau\sigma}$ are mutual inverses.

The (DR) identity $(x/y)\backslash x=y$ is $x\backslash\negthinspace\backslash (x/y)=y$, which becomes $x\,xy\mu^{\tau\sigma\tau}\,\mu^{\sigma\tau}=y$ or $\mu^{\tau\sigma\tau}*\mu^{\sigma\tau}=\epsilon$. Finally, the (DL) identity $x/(y\backslash x)=y$ is $x/(x\backslash\negthinspace\backslash y)=y$, which becomes $x\,xy\mu^{\sigma\tau}\,\mu^{\tau\sigma\tau}=y$ or $\mu^{\sigma\tau}*\mu^{\tau\sigma\tau}=\epsilon$. Thus $\mu^{\sigma\tau}$ and $\mu^{\tau\sigma\tau}$ are mutual inverses.
\end{proof}

\begin{corollary}\label{C:hyperqgp}
The six quasigroup identities $(\mathrm{SL}), (\mathrm{IL}), (\mathrm{SR}), (\mathrm{IR}), (\mathrm{DL}), (\mathrm{DR})$ of \S\ref{SS:quasigps} all take the form
\begin{equation}\label{E:hypercan}
x\,xy\mu^{\tau g}\,\mu^{g}=y
\end{equation}
for an element $g$ of $S_3$.
\end{corollary}

\subsection{Full words and parsing trees}

As described in \S\ref{SS:BasParTr}, basic quasigroup words in the free quasigroup on an alphabet
\begin{equation}\label{E:Alph4Ful}
\{a_1, a_2, \dots, a_s\}
\end{equation}
of $s$ generators are repeated concatenations of the letters under the three basic quasigroup operations $\cdot,/,\backslash$, and each word $w$ is in one-one correspondence with a \emph{basic} \emph{parsing tree} $T_w$. It is sometimes convenient to consider an extended notion of quasigroup words and parsing trees. Thus \emph{full quasigroup words} on the alphabet \eqref{E:Alph4Ful} are repeated concatenations of the generators under the full set \eqref{E:conjugat} of all quasigroup operations, both basic and opposite. Full quasigroup words are in one-one correspondence with \emph{full} \emph{parsing trees}, which are defined recursively as follows:
\begin{enumerate}
\item[$(\mathrm a)$]
For $1\le i\le s$, the full parsing tree $F_{a_i}$ is a single vertex annotated by $a_i$;
\item[$(\mathrm b)$]
For full parsing trees $F_u,F_v$ and a basic or opposite operation $\mu^g$ from the set \eqref{E:homospop}, the tree $F_{uv\,\mu^g}$ has:
\begin{enumerate}
\item[(i)]
a base annotated by $\mu^g$, along with
\item[(ii)]
$F_u$ as a left child, and $F_v$ as a right child.
\end{enumerate}
\end{enumerate}
The basic parsing trees are precisely the full parsing trees having no node that is annotated by an opposite operation $\mu^{\sigma}$, $\mu^{\tau\sigma}$, or $\mu^{\sigma\tau}$ (compare Figure~\ref{F:trialops}).

\subsection{Nodal equivalence}

A basic quasigroup word $u$ of length $n$ determines a unique basic parsing tree $T_u$ with $n-1$ nodes. The basic parsing tree represents a so-called \emph{nodal equivalence class} $\mathbf F_u$ of $2^{n-1}$ full parsing trees, sustaining a regular action of a permutation group $(S_2)^{n-1}$ known as the \emph{nodal group} of the basic quasigroup word $u$. Each $S_2$-factor of this group, indexed by a node of $T_u$, is known as the \emph{nodal subgroup} for that node. At a given node of a full parsing tree with annotating operation $\mu^g$, the non-trivial permutation of the nodal subgroup switches the two children of the node, and changes the node's annotation to $\mu^{\sigma g}$. It fixes the remainder of the tree.

The action of the nodal group of a basic quasigroup word on the nodal equivalence class $\mathbf F_u$ of full parsing trees extends to an action on the set of corresponding full quasigroup words.

\begin{example}
Consider the basic quasigroup word $(a\cdot b)/c$ in the alphabet $\{a,b,c\}$. It determines the nodal equivalence class
$$
\{
F_{ab\,\mu\,c\,\mu^{\sigma\tau\sigma}},
F_{ba\,\mu^\sigma\,c\,\mu^{\sigma\tau\sigma}},
F_{c\,ab\,\mu\,\mu^{\tau\sigma}},
F_{c\,ba\,\mu^\sigma\,\mu^{\tau\sigma}}
\}
$$
of full parsing trees, represented by the basic parsing tree $T_{(a\cdot b)/c}=F_{ab\,\mu\,c\,\mu^{\sigma\tau\sigma}}$. Similarly, the basic quasigroup word $(a\cdot b)/c$ determines the nodal equivalence class
\begin{equation}\label{E:flwdsabc}
\{
ab\,\mu\,c\,\mu^{\sigma\tau\sigma},
ba\,\mu^\sigma\,c\,\mu^{\sigma\tau\sigma},
c\,ab\,\mu\,\mu^{\tau\sigma},
c\,ba\,\mu^\sigma\,\mu^{\tau\sigma}
\}
\end{equation}
of full quasigroup words.

The regular action of the nodal group $(S_2)^2$ of the basic word $(a\cdot b)/c$ on the set \eqref{E:flwdsabc} is displayed in the Cayley diagram
$$
\xymatrix{
ab\,\mu\,c\,\mu^{\sigma\tau\sigma}
\ar@{-}[r]
\ar@{=}[d]
&
ba\,\mu^\sigma\,c\,\mu^{\sigma\tau\sigma}
\ar@{=}[d]
\\
c\,ab\,\mu\,\mu^{\tau\sigma}
\ar@{-}[r]
&
c\,ba\,\mu^\sigma\,\mu^{\tau\sigma}
}
$$
where the (involutive) action of the nodal subgroup of the internal node $\cdot$ is given by single lines, and the (involutive) action of the nodal subgroup of the root $/$ is given by double lines.
\end{example}

Each basic quasigroup word of length $n$ is nodally equivalent to any member of a set of $2^{n-1}$ full quasigroup words. Conversely, each full quasigroup word is nodally equivalent to a unique basic quasigroup word, obtained by using nodal equivalence to replace any one of the opposite operations $\mu^{\sigma}$, $\mu^{\tau\sigma}$, or $\mu^{\sigma\tau}$ by the respective basic operation $\mu$, $\mu^{\sigma\tau\sigma}$, or $\mu^{\tau}$ (compare Figure~\ref{F:trialops}).

\begin{convention}\label{CV:BaseFull}
In later sections of the paper, the generic terms \emph{quasigroup word} and \emph{parsing tree} may refer to a basic quasigroup word or basic parsing tree, or to one of its nodally equivalent full quasigroup words or full parsing trees.
\end{convention}

\subsection{Reduced words and auxiliary bivariates}

The following terminology is motivated by Evans' observation that the Diamond Lemma holds for quasigroup words \cite{Evans51} \cite[\S I.3]{Evans90}.

\begin{definition}
(a)
A (basic or full) quasigroup word is \emph{reduced} if it will not reduce further via the quasigroup identities.
\vskip 1.5mm
\noindent
(b)
A (basic or full) parsing tree representing a quasigroup word is \emph{reduced} if its corresponding quasigroup word is reduced.
\end{definition}

\begin{remark}
Note that a basic quasigroup word or parsing tree is reduced if and only if any of its nodally equivalent full quasigroup words or parsing trees is reduced.
\end{remark}

\begin{definition}\label{D:AuxBivar}
Let $s$, $a$ and $b$ be positive integers. The \emph{auxiliary bivariate} $m^s(a,b)$ denotes the number of $(a+b)$-leaf parsing trees representing reduced quasigroup words in $s$ arguments, with an $a$-leaf basic parsing tree on the left branch, a $b$-leaf basic parsing tree on the right branch, and a given (basic or opposite) quasigroup operation at the root vertex.
\end{definition}

\begin{remark}\label{R:AuxBivar}
(a)
The triality symmetry of the language of quasigroups (Figure~\ref{F:trialops}) ensures that the auxiliary bivariate $m^s(a,b)$ is invariant under any change of the choice of quasigroup operation at the root vertex of an $(a+b)$-leaf parsing tree of the type considered in Definition~\ref{D:AuxBivar}.
\vskip 1.5mm
\noindent
(b)
In particular, $m^s(a,b)=m^s(b,a)$, since the left hand side counting certain trees with $\mu^g$ at the root corresponds to the right hand side counting certain trees with the opposite operation $\mu^{\sigma g}$ at the root.
\vskip 1.5mm
\noindent
(c)
By convention, whenever one of the arguments $s,a,b$ of an auxiliary bivariate is nonpositive, the output of the auxiliary bivariate is zero.
\end{remark}

\section{Counting quasigroup words}

\subsection{Peri-Catalan Numbers}\label{SS:p-Cdefnd}

\begin{definition}
Let $n$ and $s$ be natural numbers. The $n$-th \emph{$s$-peri-Catalan number}, denoted $P_n^s$, gives the number of reduced basic quasigroup words of length $n$ in the free quasigroup on an alphabet of $s$ letters.
\end{definition}

Since there are no constants in the language of quasigroups, $P^s_n=0$ if $n=0$ or $s=0$. Now in our bookkeeping, there are $C_n$ rooted binary trees with $n$ leaves. Then over an $s$-element alphabet, each rooted binary tree may be annotated with any one of the $s$ letters at each of its $n$ leaves, along with any one of the three basic quasigroup operations at each of its $n-1$ nodes (including the root), to yield a basic parsing tree. Thus the inequality
\begin{equation}\label{eq:pcatupperlower}
P^s_n \leq 3^{n-1}s^nC_n
\end{equation}
provides an upper bound on the $n$-th $s$-peri-Catalan number. While the upper bound is exact only for $n<3$, its asymptotic significance (the so-called ``irrelevance of quasigroup identities") will be discussed in Section~\ref{S:Asympttc}.

\renewcommand{\arraystretch}{1.2}
\begin{table}[h]
\caption{The first ten peri-Catalan numbers for $s=1,2,3$.}\label{Tb:FirstTen}
\centering
\begin{tabular}{|c|c|c|c|}
\hline
$n$ & $P^1_n$ 	& $P^2_n$	& $P^3_n$\\ \hline
1	& 1			& 2			& 3			\\ %\hline
2	& 3 			& 12			& 27			\\ %\hline
3	& 12 			& 120		& 432		\\ %\hline
4	& 87 			&1,752		& 9,531		\\ %\hline
5	& 666 		& 28,224		& 233,766		\\ %\hline
6	& 5,478 		& 487,464		& 6,143,094		\\ %\hline
7	& 47,322		& 8,814,312	& 169,029,666		\\ %\hline
8	& 422,145 	& 164,734,560	& 4,808,015,253		\\ %\hline
9	& 3,859,026 	& 3,156,739,080 	& 140,243,036,202		\\ %\hline
10	& 35,967,054	& 61,689,134,928	& 4,172,008,467,726		\\ \hline
\end{tabular}
\end{table}
\renewcommand{\arraystretch}{1}

\subsection{The inductive process}\label{SS:IndProces}

Consider $1<n\in\mathbb N$ and $1\le k<n$. We may construct a quasigroup word of length $n$ by adjoining a reduced word of length $n-k$ to a reduced word of length $k$, with one of the three basic quasigroup operations as the connective. By Definition~\ref{D:AuxBivar}, the number of reduced words obtained thus will be $m^s(n-k,k)$ for each connective.

Disregarding reductions, for each $1\le k<n$, we construct $3P^s_{n-k}P^s_k$ basic quasigroup words of length $n$ in this fashion, and thus obtain
\begin{equation}\label{E:pcatbound}
P^s_n=3\sum_{k=1}^{n-1}m^s(n-k,k)\leq 3\sum_{k=1}^{n-1} P^s_{n-k}P^s_k
\end{equation}
as an upper bound on the $n$-th $s$-peri-Catalan number.

The next section describes the reductions encountered in the inductive process. There, use of Convention~\ref{CV:BaseFull} is made to invoke full quasigroup word equivalents of the basic quasigroup words as they are combined, in order to identify possible cancelations.

\subsection{Root vertex cancelations}\label{S:rootcancellations}

Let $u$ and $v$ be reduced quasigroup words, of respective lengths $k$ and $n-k$, that are connected by a quasigroup operation $\mu^g$ to form $uv\,\mu^g$ at a step of the inductive process, as indicated by the outer level of boxes in Figure~\ref{F:RotVerCn}.

\begin{figure}[hbt]
\begin{picture}{(250,150)}

\put(21,20){$\mu^g$}
\put(26,22){\oval(20,25)}
\put(26,35){\line(0,1){49}}
\put(5,130){length $k$}
\put(6,100){$\boxed{\rule{4mm}{0mm}\rule[-4mm]{0mm}{10mm}u\rule{4mm}{0mm}}$}
\put(37,20){\line(1,0){72}}

\put(175,130){length $n-k$}
\put(110,100){$\boxed{\rule{30mm}{0mm}\rule[-34mm]{0mm}{40mm}v\rule{10mm}{0mm}}$}
\put(170,50){length $n-2k$}
\put(181,20){$\boxed{\rule{4mm}{0mm}\rule[-4mm]{0mm}{10mm}v'\rule{4mm}{0mm}}$}

\put(120,105){length $k$}
\put(121,75){$\boxed{\rule{4mm}{0mm}\rule[-4mm]{0mm}{10mm}u\rule{4mm}{0mm}}$}
\put(134,20){$\mu^{\tau g}$}
\put(141,22){\oval(20,25)}
\put(141,35){\line(0,1){25}}

\put(152,20){\line(1,0){28}}

\end{picture}
\caption{Root vertex cancelation}\label{F:RotVerCn}
\end{figure}
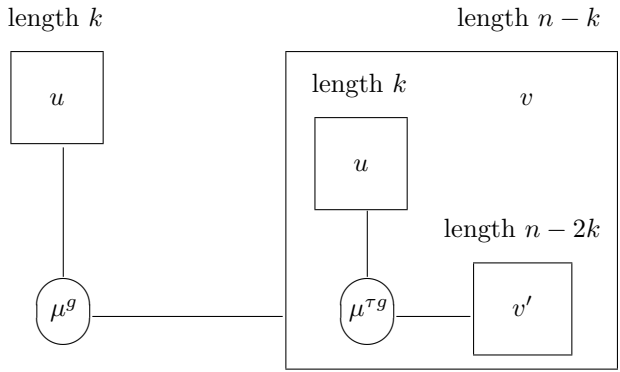

\begin{proposition}\label{P:cancel_formats}
During the assembly of $uv\,\mu^g$ within the inductive process, with reduced word $u$ of length $k$ and reduced word $v$ of length $n-k$, cancelation occurs if and only if there is a (necessarily reduced) word $v'$ of length $n-2k$ such that $v=uv'\,\mu^{\tau g}$.
\end{proposition}

\begin{proof}
The unique cancelations available are of the form $u\,uv'\mu^{\tau g}\,\mu^{g}=v'$ described in Corollary~\ref{C:hyperqgp}. Since the word $v=uv'\,\mu^{\tau g}$ is reduced, it follows that the subword $v'$ is also reduced.
\end{proof}

The cancelation condition in Proposition~\ref{P:cancel_formats} is illustrated by the inner level of boxes on the right hand side of Figure~\ref{F:RotVerCn}.

\begin{proposition}\label{P:auxbivreductions}
Consider $1<n\in\mathbb N$ and $1\le k<n$.
\begin{enumerate}
\item[$(\mathrm a)$]
The number of cancelations incurred during the inductive process when a word of length $k$ is connected to a word of length $n-k$ by a given operation $\mu^g$ from the set \eqref{E:homospop} is $m^s(n-2k,k)$. In particular, no cancelation occurs when $n=2k$.
\item[$(\mathrm b)$]
The formula
\begin{equation}\label{E:PnkPkmnk}
m^s(n-k,k)=P_{n-k}^sP_k^s-m^s(n-2k,k)
\end{equation}
holds.
\end{enumerate}
\end{proposition}

\begin{proof}
(a)
A reduced word $u$ of length $k$ will appear at most once in a reduced word $v$ of length $n-k$ with $\mu^{\tau g}$ in the root vertex. The number of such configurations (illustrated by the right-hand box in Figure~\ref{F:RotVerCn}) is given by $m^s(n-2k,k)$.
\vskip 1.5mm
\noindent
(b)
The formula \eqref{E:PnkPkmnk} results from (a) and Definition~\ref{D:AuxBivar}.
\end{proof}

\subsection{The Euclidean Algorithm}\label{SS:EucliAlg}

The next section presents the main recursive formula for peri-Catalan numbers. The formula is structured around the Euclidean Algorithm, using notation established in this paragraph following \cite[\S1.6]{ITAA}.

For $1<n\in\mathbb N$ and $1\leq k\leq n-1$, let $r^k_{-1}=n$ and $r^k_{0}=k$. Consider the quotients $q^k_{l}$ and remainders $r^k_{l}$ for $1\le l\le L_{k+1}$ as given below, resulting from calls to the Division Algorithm in the computation of $\gcd(n,k)$ by the Euclidean Algorithm:
\begin{align}\label{E:QuotnRem}
r^k_{-1}&=q^k_{1}r^k_{0}+r^k_{1}\,,
%\\
\ldots\,,
%\\
r^k_{l-2}=q^k_{l}r^k_{l-1}+r^k_{l}\,,
%\\
\ldots\,,
%\\
r^k_{L_{k}-1}=q^k_{L_{k}+1}r^k_{L_k}+r^k_{L_{k}+1}\,.
\end{align}
Here, $r^k_{L_{k}+1}=0$ and $\gcd(n,k)=r^k_{L_k}$. Define quantities $\epsilon^k_l$ for $0\le l\le L_{k}$ recursively by
\begin{equation}\label{E:RecurEps}
\epsilon^k_0=1
\quad
\mbox{ and }
\quad
\epsilon^k_{l+1}=\epsilon^k_{l}+q^k_{l+1} \, .
\end{equation}
These quantities, along with the quotients $q^k_{l}$ and remainders $r^k_{l}$, are used in the following section.

\subsection{The recursive formula for peri-Catalan numbers}\label{SS:RecProPC}

\begin{lemma}\label{L:BaseCase}
Using the notation $r^k_{-1}=n$, $r^k_{0}=k$, $r^k_{-1}=q^k_{1}r^k_{0}+r^k_{1}$, and $\epsilon^k_0=1$ from \S\ref{SS:EucliAlg}, the formula
\begin{align}\label{E:BaseCase}
m^s(n-k,k)
=(-1)^{q^k_{1}-1}m^s(r^k_{0},r^k_{1})
+\sum_{j^k_{0}=1}^{q^k_{1}-1} (-1)^{\epsilon^k_{0}+j^k_{0}} P^s_{r^k_{-1}-j^k_{0}r^k_{0}}P^s_{r^k_{0}}
\end{align}
holds for $1<n\in\mathbb N$ and $1\leq k\leq n-1$.
\end{lemma}

\begin{proof}
It will be shown, by induction on $i$, that
\begin{equation}\label{E:BasCasIH}
m^s(n-k,k)
=(-1)^{i}m^s\big(r^k_{0},r^k_{-1}-(i+1)r^k_{0}\big)
+\sum_{j^k_{0}=1}^{i} (-1)^{\epsilon^k_{0}+j^k_{0}} P^s_{r^k_{-1}-j^k_{0}r^k_{0}}P^s_{r^k_{0}}
\end{equation}
for $0\le i< q^k_1$. Note that \eqref{E:BasCasIH} for $i=q^k_1-1$ yields \eqref{E:BaseCase}. On the other hand, the base of the induction, namely \eqref{E:BasCasIH} with $i=0$, is given by Remark~\ref{R:AuxBivar}(b).

Now suppose that the induction hypothesis \eqref{E:BasCasIH} holds for $0\le i<q^k_1-1$. Then
\begin{align*}
m^s(n-k,k)
&
=(-1)^{i}m^s\big(r^k_{0},r^k_{-1}-(i+1)r^k_{0}\big)
+\sum_{j^k_{0}=1}^{i} (-1)^{\epsilon^k_{0}+j^k_{0}} P^s_{r^k_{-1}-j^k_{0}r^k_{0}}P^s_{r^k_{0}}
\\
&
=(-1)^{i}\Big[P^s_{r^k_{-1}-(i+1)r^k_{0}}P^s_{r^k_{0}}
-m^s\big(r^k_{0},r^k_{-1}-(i+2)r^k_{0}\big)\Big]
\\
&
\rule{53mm}{0mm}
+\sum_{j^k_{0}=1}^{i} (-1)^{\epsilon^k_{0}+j^k_{0}}
P^s_{r^k_{-1}-j^k_{0}r^k_{0}}P^s_{r^k_{0}}
\\
&
=(-1)^{i+1}m^s\big(r^k_{0},r^k_{-1}-(i+2)r^k_{0}\big)
\\
&
\rule{6mm}{0mm}
+(-1)^{\epsilon^k_0+(i+1)}P^s_{r^k_{-1}-(i+1)r^k_{0}}P^s_{r^k_{0}}
+\sum_{j^k_{0}=1}^{i} (-1)^{\epsilon^k_{0}+j^k_{0}}
P^s_{r^k_{-1}-j^k_{0}r^k_{0}}P^s_{r^k_{0}}
\\
&
=(-1)^{i+1}m^s\big(r^k_{0},r^k_{-1}-(i+2)r^k_{0}\big)
+\sum_{j^k_{0}=1}^{i+1} (-1)^{\epsilon^k_{0}+j^k_{0}}
P^s_{r^k_{-1}-j^k_{0}r^k_{0}}P^s_{r^k_{0}}
\end{align*}
by \eqref{E:PnkPkmnk} and Remark~\ref{R:AuxBivar}(b), as required for the induction step.
\end{proof}

\begin{corollary}\label{C:BaseCase}
The formula
\begin{align}\label{E:BaseCass}
m^s(n-k,k)
=(-1)^{\epsilon^k_{1}}m^s(r^k_{0},r^k_{1})
+\sum_{j^k_{0}=1}^{q^k_{1}-1} (-1)^{\epsilon^k_{0}+j^k_{0}} P^s_{r^k_{-1}-j^k_{0}r^k_{0}}P^s_{r^k_{0}}
\end{align}
holds for $1<n\in\mathbb N$ and $1\leq k\leq n-1$.
\end{corollary}

\begin{proof}
Note that $(-1)^{\epsilon^k_{1}}=(-1)^{1+q^k_{1}}=(-1)^{q^k_{1}-1}$ by \eqref{E:RecurEps}, so \eqref{E:BaseCass} is a restatement of \eqref{E:BaseCase}.
\end{proof}

\begin{lemma}\label{L:InduStep}
Using the notation of \S\ref{SS:EucliAlg}, suppose that $m^s(n-k,k)$ is given by
\begin{align}\label{E:IndHypoP}
K(n,k)+
(-1)^{\epsilon^k_{l}}m^s(r^k_{l-1},r^k_{l})
+
\sum_{i=0}^{l-1}
\sum_{j^k_{i}=0}^{q^k_{i+1}-1} (-1)^{\epsilon^k_{i}+j^k_{i}} P^s_{r^k_{i-1}-j^k_{i}r^k_{i}}P^s_{r^k_{i}}
\end{align}
for $1<n\in\mathbb N$, $1\leq k\leq n-1$, and $0<l\le L_k$, with an additive term $K(n,k)$.
Then $m^s(n-k,k)=$
\begin{align}\label{E:IndConcP}
K(n,k)
+(-1)^{\epsilon^k_{l+1}}m^s(r^k_{l},r^k_{l+1})
+
\sum_{i=0}^{l}
\sum_{j^k_{i}=0}^{q^k_{i+1}-1} (-1)^{\epsilon^k_{i}+j^k_{i}} P^s_{r^k_{i-1}-j^k_{i}r^k_{i}}P^s_{r^k_{i}}
\end{align}
follows.
\end{lemma}

\begin{proof}
It will be shown, by induction on $h$, that
\begin{align}\label{E:InStepIH}
&m^s(n-k,k)=K(n,k)
+(-1)^{\epsilon^k_{l}+h}m^s\big(r^k_{l},r^k_{l-1}-hr^k_{l}\big)
\\ \notag
&
+\sum_{j^k_{l}=0}^{h-1} (-1)^{\epsilon^k_{l}+j^k_{l}} P^s_{r^k_{l-1}-j^k_{l}r^k_{l}}P^s_{r^k_{l}}
+
\sum_{i=0}^{l-1}
\sum_{j^k_{i}=0}^{q^k_{i+1}-1}(-1)^{\epsilon^k_{i}+j^k_{i}} P^s_{r^k_{i-1}-j^k_{i}r^k_{i}}P^s_{r^k_{i}}
\end{align}
for $0\le h\le q^k_{1+1}$. Recalling \eqref{E:RecurEps}, note that \eqref{E:InStepIH} with $h=q^k_{l+1}$ yields the expression \eqref{E:IndConcP} for $m^s(n-k,k)$. On the other hand, the base of the induction, namely \eqref{E:InStepIH} with $h=0$, is given by \eqref{E:IndHypoP}.

Starting from \eqref{E:InStepIH} with $h<q^k_{1+1}$, the induction step is given as $m^s(n-k,k)=$
\begin{align*}
&
K(n,k)
+(-1)^{\epsilon^k_{l}+h}m^s\big(r^k_{l},r^k_{l-1}-hr^k_{l}\big)
\\
&
+\sum_{j^k_{l}=0}^{h-1}(-1)^{\epsilon^k_{l}+j^k_{l}} P^s_{r^k_{l-1}-j^k_{l}r^k_{l}}P^s_{r^k_{l}}
+
\sum_{i=0}^{l-1}
\sum_{j^k_{i}=0}^{q^k_{i+1}-1}(-1)^{\epsilon^k_{i}+j^k_{i}} P^s_{r^k_{i-1}-j^k_{i}r^k_{i}}P^s_{r^k_{i}}
\\
&
=K(n,k)
+(-1)^{\epsilon^k_{l}+h}
\Big[
P^s_{r^k_{l-1}-hr^k_{l}}P^s_{r^k_{l}}
-m^s\big(r^k_{l},r^k_{l-1}-(h+1)r^k_{l}\big)
\Big]
\\
&
+\sum_{j^k_{l+1}=0}^{h-1} (-1)^{\epsilon^k_{l}+j^k_{l}} P^s_{r^k_{l-1}-j^k_{l}r^k_{l}}P^s_{r^k_{l}}
+
\sum_{i=0}^{l-1}
\sum_{j^k_{i}=0}^{q^k_{i+1}-1}(-1)^{\epsilon^k_{i}+j^k_{i}} P^s_{r^k_{i-1}-j^k_{i}r^k_{i}}P^s_{r^k_{i}}
\\
&
=K(n,k)
+(-1)^{\epsilon^k_{l}+(h+1)}m^s\big(r^k_{l},r^k_{l-1}-(h+1)r^k_{l}\big)
\\ \notag
&
+\sum_{j^k_{l}=0}^{h} (-1)^{\epsilon^k_{l}+j^k_{l}} P^s_{r^k_{l-1}-j^k_{l}r^k_{l}}P^s_{r^k_{l}}
+
\sum_{i=0}^{l-1}
\sum_{j^k_{i}=0}^{q^k_{i+1}-1} (-1)^{\epsilon^k_{i}+j^k_{i}} P^s_{r^k_{i-1}-j^k_{i}r^k_{i}}P^s_{r^k_{i}}
\end{align*}
by \eqref{E:PnkPkmnk} and Remark~\ref{R:AuxBivar}(b).
\end{proof}

\begin{proposition}\label{P:m_unpacking}
Let $1<n\in\mathbb N$ and $1<k<n$. Then $m^s(n-k,k)$ is specified by
\begin{align}\label{E:mnkUnpak}
P^s_{r^k_{-1}}P^s_{r^k_0}
+
\sum_{i=0}^{L_k}
\sum_{j^k_{i}=0}^{q^k_{i+1}-1} (-1)^{\epsilon^k_{i}+j^k_{i}} P^s_{r^k_{i-1}-j^k_{i}r^k_{i}}P^s_{r^k_{i}}
\end{align}
in the notation of \S\ref{SS:EucliAlg}.
\end{proposition}

\begin{proof}
It will be proved, by induction on $l$, that $m^s(n-k,k)=$
\begin{align}\label{E:InHyProp}
P^s_{r^k_{-1}}P^s_{r^k_0}
+(-1)^{\epsilon^k_{l+1}}m^s(r^k_{l},r^k_{l+1})
+
\sum_{i=0}^{l}
\sum_{j^k_{i}=0}^{q^k_{i+1}-1} (-1)^{\epsilon^k_{i}+j^k_{i}} P^s_{r^k_{i-1}-j^k_{i}r^k_{i}}P^s_{r^k_{i}}
\end{align}
holds for $0\le l\le L_k$. Note that \eqref{E:InHyProp} is just \eqref{E:IndConcP} with $$
K(n,k)
=P^s_{r^k_{-1}}P^s_{r^k_0}
=P^s_{n}P^s_k
$$
as the additive term.

The induction basis is given by \eqref{E:BaseCass} in Corollary~\ref{C:BaseCase}. The induction step is given by Lemma~\ref{L:InduStep}. Then the proposition is given by \eqref{E:InHyProp} with $l=L_k$, since $r^k_{L_k+1}=0$ and thus $m^s(r^k_{L_k},r^k_{L_k+1})=0$.
\end{proof}

\begin{corollary}\label{C:m_unpacking}
Let $1<n\in\mathbb N$ and $1<k<n$. Then $m^s(n-k,k)$ is specified by
\begin{align}\label{E:mnkUnpal}
\sum_{j^k_{0}=1}^{q^k_{1}-1} (-1)^{\epsilon^k_{0}+j^k_{0}} P^s_{r^k_{-1}-j^k_{0}r^k_{0}}P^s_{r^k_{0}}
+
\sum_{i=1}^{L_k}
\sum_{j^k_{i}=0}^{q^k_{i+1}-1} (-1)^{\epsilon^k_{i}+j^k_{i}} P^s_{r^k_{i-1}-j^k_{i}r^k_{i}}P^s_{r^k_{i}}
\end{align}
in the notation of \S\ref{SS:EucliAlg}.
\end{corollary}

\begin{proof}
Note that $P^s_{r^k_{-1}}P^s_{r^k_0}$ cancels the $i=0$, $j^k_i=0$ summand in the term
$$
\sum_{i=0}^{L_k}
\sum_{j^k_{i}=0}^{q^k_{i+1}-1} (-1)^{\epsilon^k_{i}+j^k_{i}} P^s_{r^k_{i-1}-j^k_{i}r^k_{i}}P^s_{r^k_{i}}
$$
of \eqref{E:mnkUnpak} in Proposition~\ref{P:m_unpacking}. Thus \eqref{E:mnkUnpal} specifies $m^s(n-k,k)$ in the same way that \eqref{E:mnkUnpak} does.
\end{proof}

As there are no constants in the language of quasigroups, $P^s_0=0$. Corollary~\ref{C:m_unpacking} and \eqref{E:pcatbound} then yield the main result.

\begin{theorem}\label{T:ExactPsn}
For $1<n\in\mathbb N$, the $n$-th $s$-peri-Catalan number $P^s_n$ is given by
\begin{equation}\label{E:FormuPsn}
3\sum_{k=1}^{n-1}
\bigg\{
\sum_{j^k_{0}=1}^{q^k_{1}-1} (-1)^{\epsilon^k_{0}+j^k_{0}} P^s_{r^k_{-1}-j^k_{0}r^k_{0}}P^s_{r^k_{0}}
+
\sum_{i=1}^{L_k}
\sum_{j^k_{i}=0}^{q^k_{i+1}-1} (-1)^{\epsilon^k_{i}+j^k_{i}} P^s_{r^k_{i-1}-j^k_{i}r^k_{i}}P^s_{r^k_{i}}
\bigg\}
\end{equation}
in the notation of \S\ref{SS:EucliAlg}.
\end{theorem}

\begin{example}
We compute the fourth 2-peri-Catalan number as follows. For $k=1, 2, 3$, the table
$$
\begin{tabular}{c|c|c}
$k=1$		& $k=2$			& $k=3$ 	\\ \hline
$4=4\cdot 1+0$ & $4=2\cdot 2+0$ 	& $4=1\cdot 3 + 1$\\
			&				& $3=3\cdot 1+0$
\end{tabular}
$$
displays the relevant runs of the Euclidean Algorithm, yielding $L_1=0$, $L_2=0$, and $L_3=1$ in the notation of \S\ref{SS:EucliAlg}.

We then compute $P^2_4$ as
\begin{align*}
P^2_4 =&3\sum_{k=1}^{3} \bigg\{\sum_{j^k_{0}=1}^{q^k_{1}-1} (-1)^{\epsilon^k_{0}+j^k_{0}} P^s_{r^k_{-1}-j^k_{0}r^k_{0}}P^s_{r^k_{0}} + \sum_{i=1}^{L_k} \sum_{j^k_{i}=0}^{q^k_{i+1}-1} (-1)^{\epsilon^k_{i}+j^k_{i}} P^s_{r^k_{i-1}-j^k_{i}r^k_{i}}P^s_{r^k_{i}}
\bigg\}\\
	=& 3\bigg\{ \sum_{j^1_0=1}^3 (-1)^{1+j^1_0}P^2_{4-j^1_0}P^2_1 + \sum_{j^2_0=1}^1 (-1)^{1+j^2_0} P^2_{4-2j^2_0}P^2_2 \\
	&+ \sum_{j^3_0=1}^0 (-1)^{1+j_0^3}P^2_{4-3j^3_1}P^2_3 +\sum_{j^3_1=0}^2(-1)^{\epsilon_1^3+j^3_1}P^2_{3-j^3_1}P^2_1\bigg\}\\
	=&3\bigg\{(P^2_3P^2_1-P^2_2P_1^2+P^2_1P^2_1) + P^2_2P^2_2 +(P^2_3P^2_1 -P^2_2P^2_1+P^2_1P^2_1)\bigg\}\\
	=&1,752
\end{align*}
to obtain the value displayed in Table~\ref{Tb:FirstTen}.

\end{example}

\section{Asymptotic behavior}\label{S:Asympttc}

Theorem~\ref{T:ExactPsn}, and in particular the form of the recursive expression \eqref{E:FormuPsn} for the $n$-th $s$-peri-Catalan number $P^s_n$, do not readily lend themselves to an analytical asymptotic estimate comparable with the estimate
$$
C_n\sim\frac{4^n}{\sqrt{\pi n^3}}
$$
for the Catalan numbers that results from their generating function \cite[I.2.4(33)]{FlajSedg}. Instead, we will sketch out a conjectured alternative approach, and present some numerical data in its support. The approach depends on the close relationship between the Catalan numbers and the peri-Catalan numbers.

\subsection{Asymptotic irrelevance of quasigroup identities}

The foundation of our approach is the following.

\begin{conjecture}[Asymptotic irrelevance of quasigroup identities]\label{Cj:AsIrQpId}
In the large, cancelation resulting from the quasigroup identities has a negligible effect on the asymptotic behavior of the peri-Catalan numbers $P^s_n$.
\end{conjecture}

In Conjecture~\ref{Cj:AsIrQpId}, the phrase ``in the large" refers informally to large values of $s$ for almost all values of $n$.
Heuristically, referring to Figure~\ref{F:RotVerCn}, cancelation by a quasigroup identity would require the assignment of the $k$ variables in the first copy of the word $u$ to match exactly with the assignment of the $k$ variables in the second copy of $u$. Such an event should be relatively unlikely when there is a large total of $s$ variables available.

\subsection{The effect of asymptotic irrelevance}\label{SS:EfOAsIrr}

We conjecture that
\begin{equation}\label{E:Lim13612}
\lim_{s\to\infty}
\lim_{n\to\infty}
\frac{\log P^s_n}{\log C_n+n\log 3s-\log3}
=1
\end{equation}
under the assumption of asymptotic irrelevance, and in particular, we conjecture that the indicated limits exist.

\begin{center}
\begin{figure}[htb]
\caption{Plots of $\log P^s_n/\big(\log C_n+n\log 3s-\log3\big)$ for $s=1,3,6,12$.}\label{F:Lim13612}
\resizebox{120mm}{90mm}{\includegraphics{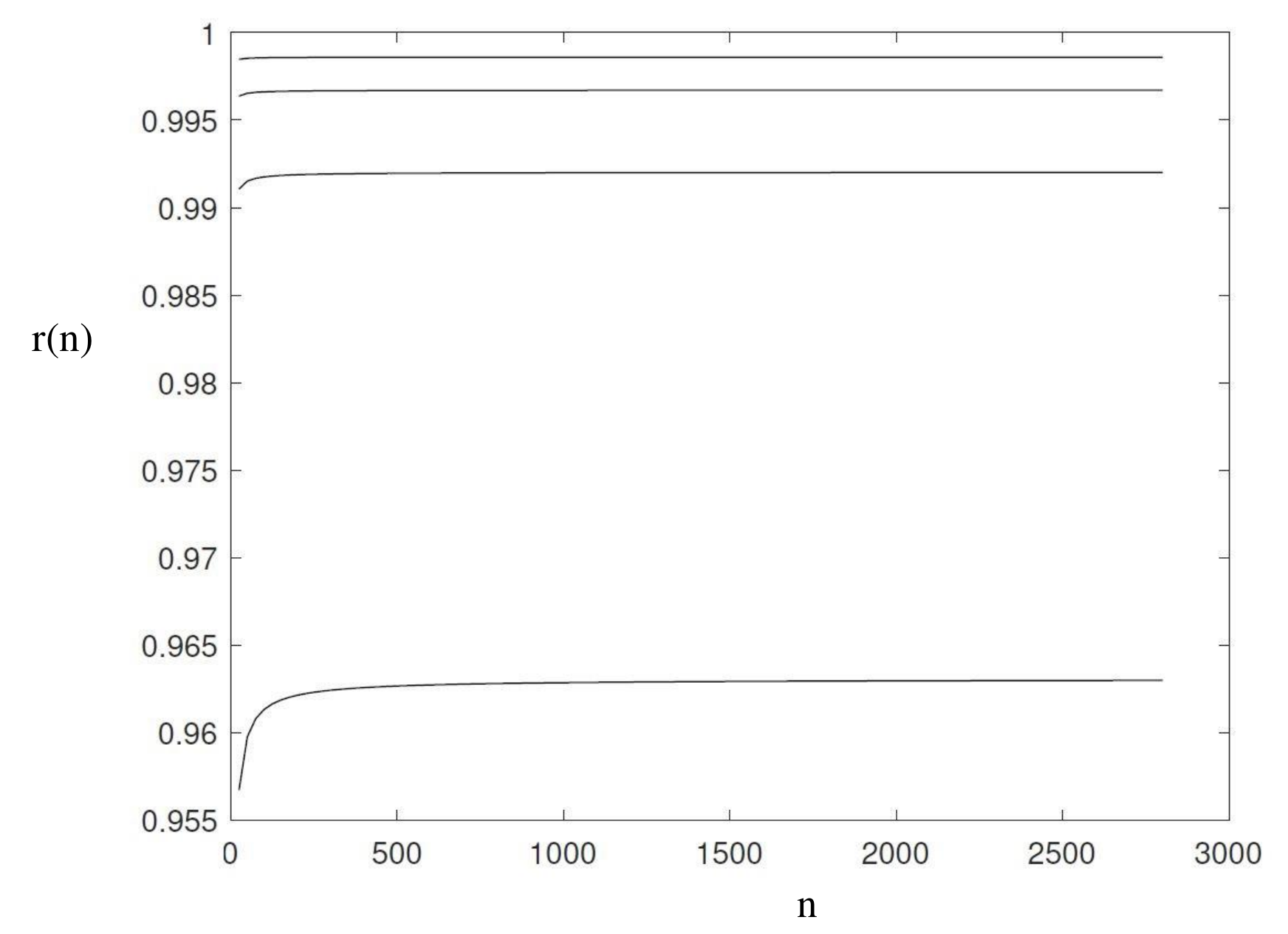}}
\end{figure}
\end{center}

As initial evidence, we offer Figure~\ref{F:Lim13612}, which displays plots of the quotients
\begin{equation}\label{E:PsnCn3s3}
\frac{\log P^s_n}{\log C_n+n\log 3s-\log3}
\end{equation}
for $s=1,3,6,12$, and $n\le 2800$ on the horizontal axis. The curves for the respective $s$-values appear in order from the bottom to the top. Based on this evidence, we conjecture that, for a given value of $s$, the sequence of terms \eqref{E:PsnCn3s3} increases monotonically with $n$. Now the denominator in \eqref{E:PsnCn3s3} is the logarithm of the upper bound from \eqref{eq:pcatupperlower}, the total number of words of length $n$, on $s$ generators, in the language of quasigroups. The numerator in \eqref{E:PsnCn3s3} is the logarithm of $P^s_n$, the total number of reduced words of length $n$ on $s$ generators. Thus the terms \eqref{E:PsnCn3s3} are bounded above by $1$, and under our conjecture of monotonicity, the inner limit in \eqref{E:Lim13612} would exist.

Now again, based on evidence as initially exhibited in Figure~\ref{F:Lim13612}, and then with greater detail as in \S\ref{SS:DepGeNum} below, we conjecture that the sequence of these inner limits increases with $s$. The upper bound of $1$ still applies, and thus the outer limit in \eqref{E:Lim13612} would exist. That this limit is $1$ is essentially the content of the conjecture of asymptotic irrelevance.

While Figure~\ref{F:Lim13612} is useful for exhibiting qualitative features, in particular the monotonicities, it is less useful for gaining quantitative insight. Thus it is worth mentioning a second numerical experiment, where the values of $\log P^{12}_n-\log C_n$, over intervals of $n$ up to $n=2800$, were analyzed in Matlab. Linear regressions of these values were obtained as $3.576n-1.102$ (with very tight error bars). The slope of the linear regression matches $\log 36\simeq 3.583$ to three significant digits, while the (negated) intercept matches $\log 3\simeq 1.099$ to the same accuracy.

\subsection{Dependence on the number of generators}\label{SS:DepGeNum}

We present approximations to the values of
\begin{equation}\label{E:PropOCan}
1-\lim_{n\to\infty}\frac{\log P^s_n}{\log C_n+n\log 3s-\log3}
\end{equation}
for $1\le s\le 100$ in Table~\ref{Tb:1minuslm}. The entries, which provide a numerical measure for the effect of cancelation by the quasigroup identities, are presented in the format ``$M$e-$N$" standing for $M\times10^{-N}$. The quotient
$$
\frac{\log P^s_{2000}}{\log C_{2000}+2000\log 3s-\log3}
$$
is used as a proxy for the limit. The function of $s$ represented by Table~\ref{Tb:1minuslm} was fitted to the rational function
$$
\frac{0.01929}{s-0.4811}
$$
by Matlab, with very tight 95\% confidence bounds. On this basis, noting that the logarithm of the golden ratio is $0.4812$, we formulate the following, which subsumes the unlabeled conjectures of \S\ref{SS:EfOAsIrr}.

\begin{conjecture}\label{Cj:DepGeNum}
We have the approximate values
$$
\lim_{n\to\infty}
\frac{\log P^s_n}{\log C_n+n\log 3s-\log3}
\simeq 1-\frac{0.019\dots}{s-\log\big((1+\sqrt5)/2\big)}
$$
for all generator counts $s$. In particular, the limit exists.
\end{conjecture}

We do not currently have a conjectured interpretation for the numerator $0.019\dots$.

\renewcommand{\arraystretch}{1.2}
\begin{table}[hbt]
\caption{Decreasing impact of quasigroup identities as $s$ increases}\label{Tb:1minuslm}
\begin{tabular}{|c|c|c|c|}
\hline
\hline
$1\le s\le 25$ &$26\le s\le 50$ &$51\le s\le 75$ &$76\le s\le 100$\\ \hline\hline
370e-4
&   561e-6
 &  255e-6
  & 161e-6
\\
\hline
137e-4
&   536e-6
 &  250e-6
  & 159e-6
\\
\hline
800e-5
&   514e-6
 &  244e-6
  & 157e-6
\\
\hline
551e-5
&   493e-6
 &  239e-6
  & 154e-6
\\
\hline
415e-5
&   474e-6
 &  234e-6
  & 152e-6
\\
\hline
330e-5
&   456e-6
 &  229e-6
  & 150e-6
\\
\hline
272e-5
&   439e-6
 &  225e-6
  & 148e-6
\\
\hline
231e-5
&   424e-6
 &  220e-6
  & 146e-6
  \\
  \hline
200e-5
&   409e-6
 &  216e-6
  & 144e-6
  \\
  \hline
176e-5
&   396e-6
 &  212e-6
  & 142e-6
  \\
  \hline
157e-5
&   383e-6
 &  208e-6
  & 140e-6
  \\
  \hline
141e-5
&   371e-6
 &  204e-6
  & 138e-6
  \\
  \hline
128e-5
&   359e-6
 &  200e-6
  & 136e-6
  \\
  \hline
  117e-5
&   349e-6
 &  197e-6
  & 135e-6
  \\
  \hline
    108e-5
&   340e-6
 &  193e-6
  & 133e-6
  \\
  \hline
    997e-6
&   329e-6
 &  190e-6
  & 131e-6
  \\
  \hline
    928e-6
&   320e-6
 &  186e-6
  & 130e-6
  \\
  \hline
    867e-6
&   311e-6
 &  183e-6
  & 128e-6
  \\
  \hline
    813e-6
&   303e-6
 &  180e-6
  & 126e-6
  \\
  \hline
    765e-6
&   295e-6
 &  177e-6
  & 124e-6
  \\
  \hline
    722e-6
&   288e-6
 &  174e-6
  & 123e-6
  \\
  \hline
    683e-6
&   281e-6
 &  172e-6
  & 122e-6
  \\
  \hline
    648e-6
&   274e-6
 &  169e-6
  & 121e-6
  \\
  \hline
    616e-6
&   268e-6
 &  166e-6
  & 119e-6
  \\
  \hline
    587e-6
&   261e-6
 &  164e-6
  & 118e-6
  \\
  \hline
\end{tabular}
\end{table}
\renewcommand{\arraystretch}{1}

\end{document}